\documentclass{amsart}
\usepackage{geometry}
\usepackage{amssymb,latexsym,enumitem,bbm}
\usepackage[noadjust]{cite}
\usepackage{setspace}
\usepackage{color}
\usepackage{hyperref}
\usepackage{bbm}

\newcommand{\R}{\mathbb{R}}

\newcommand{\F}{\mathcal{F}}

\newcommand{\Hh}{\mathcal{H}}

\newcommand{\Sc}{\mathcal{S}}

\newcommand{\Exp}{\mathbb{E}}

\newcommand{\h}{\hat}

\newcommand{\inpr}[3][]{\left\langle#2 \,,\, #3\right\rangle_{#1}}

\numberwithin{equation}{section}

\newtheorem{theorem}{Theorem}[section]

\newtheorem{remark}[theorem]{Remark}

\newtheorem{definition}[theorem]{Definition}

\newtheorem{example}[theorem]{Example}

\allowdisplaybreaks

\title{Weak Solutions of SPDEs in the space of Tempered distributions}

\author{Suprio Bhar}
\address{Suprio Bhar, Department of Mathematics and Statistics, Indian Institute of Technology Kanpur, Kalyanpur, Kanpur - 208016, India.}
\email{suprio@iitk.ac.in}

\author{Barun Sarkar}
\address{Barun Sarkar, Department of Mathematics, Indian Institute of Technology Madras, Chennai - 600036, India.}
\email{barun@iitm.ac.in}

\begin{document}
  

 \begin{abstract}
 In this article, we construct weak solutions for a class of Stochastic PDEs in the space of tempered distributions via Girsanov's theorem. It is to be noted that our drift and diffusion coefficients $(L,A)$ of the considered Stochastic PDE satisfy a Monotonicity type inequality, rather than Lipschitz conditions. As such, we can not follow the usual infinite dimensional analysis as described in \cite[sections 10.2 and 10.3]{MR3236753}. Instead, we exploit related SDEs to obtain our desired result, and we point out an important observation that the same Novikov condition is used in changing the Brownian motion in both the SDEs and the Stochastic PDEs.
 \end{abstract}

  \maketitle

\section{Introduction}\label{S1:intro}
Fix $T > 0$ and let $(\Omega, \F, \{\F_t\}_{t \in [0, T]}, P)$ be a complete filtered probability space, satisfying the usual conditions. In this article, we study weak and strong solutions of the following stochastic partial differential equation (SPDE) in $\Sc'$, the space of tempered distributions on $\R^d$, for $t\in[0,T]$
 \begin{align}\label{main-spde}
 \begin{split}
 & dX_t = L(X_t)\, dt + A(X_t) \cdot dB_t,\\
 & X_0=\phi, 
 \end{split}
 \end{align}
 where $B = \{B_t\}_{0\leq t\leq T}$ is a given $d$-dimensional standard Brownian motion with respect to the filtration $\{\F_t\}_{0\leq t\leq T}$, $\phi$ is an $\Sc^\prime$ valued deterministic initial condition and $A:=(A_1,\cdots,A_d)$, $L,A_j:\Sc'\to\Sc'$ are nonlinear operators such that for $y\in\Sc'$
 \begin{align}\label{defn-LA1}
 \begin{split}
 & L(y) := \frac{1}{2} \sum_{i,j=1}^d (\sigma(y)\sigma^t(y))_{ij}\,  \partial^2_{ij} y - \sum_{i=1}^d b_i(y)\, \partial_i y,\\
 & A_j(y) := - \sum_{i=1}^d \sigma_{ij}(y)\, \partial_i y,
 \end{split}
 \end{align}
where $\sigma:\Sc^\prime \to \R^{d \times d}, b: \Sc^\prime \to \R^{d}$, with the components denoted by $\sigma_{ij}, b_i, i, j = 1, 2, \cdots, d$ and $\sigma^t$ denotes the transpose of $\sigma$. Descriptions of the topology on $\Sc$, the Schwartz space on $\R^d$ and definitions of the notions of solutions of the SPDEs considered above have been recalled in subsection \ref{S2:topology} and subsection \ref{S2:setup}, respectively. For the notions of weak and strong solutions, we refer to \cite{MR3236753, MR2560625, MR532498} and the references therein.

In \cite{MR3063763}, local strong solutions to the SPDE \eqref{main-spde} were shown to arise from the local strong solutions from certain associated stochastic differential equations (SDEs). We recall this correspondence in subsection \ref{S2:setup}. Note that the same Brownian motion appears in both the equations, the SPDE as well as the SDE.  The correspondence holds, provided the pair of operators $(L, A)$ satisfy the Monotonicity inequality.

Using the correspondence mentioned above, we apply the finite dimensional Girsanov's Theorem to change the drift terms in both the SPDE and the SDE. Consequently, we are able to use the same Novikov's integrability condition in our arguments.  This leads to the existence of a weak solution as well as uniqueness in law for the modified SPDE, with the new Brownian motion arising from the finite dimensional Girsanov's Theorem. Since we are working with Hermite-Sobolev space valued processes (see Section \ref{S2:prelims}), which are driven by a finite dimensional Brownian motion, as a consequence neither we can use infinite-dimensional approach as in \cite[sections 10.2 and 10.3]{MR3236753}, nor the finite-dimensional results as in \cite{MR2001996}. It is to be noted that in \cite{MR3236753}, the noise is Hilbert valued and in \cite{MR2001996}, the process in consideration is finite-dimensional. The main results of this article have been discussed in Section \ref{S3:main-results} and some applications have been mentioned in Section \ref{S4:Applications}. Note that our assumption does not include any Lipschitz continuity of $(L, A)$.

\section{Preliminaries}\label{S2:prelims}

\subsection{Topology on Schwartz space}\label{S2:topology}
Let $\Sc$ denote the space of real valued rapidly decreasing smooth functions on $\R^d$, with the topology given by L. Schwartz (\cite{MR1681462, MR2296978}). Note that its dual is $\Sc^\prime$ (see \cite{MR771478}). Let $\mathbb{Z}^d_+:=\{n=(n_1,\cdots, n_d): \; n_i \text{ non-negative integers}\}$. If $n\in\mathbb{Z}^d_+$, we define $|n|:=n_1+\cdots+n_d$. For $p \in \R$, consider the increasing norms $\|\cdot\|_p$, defined by the inner
products
\begin{equation}
\langle f,g\rangle_p:=\sum_{n\in\mathbb{Z}^d_+}(2|n|+d)^{2p}\langle f,h_n\rangle\langle g,h_n\rangle,\ \ \ f,g\in\Sc.
\end{equation}
In the above equation, $\{h_n: n\in\mathbb{Z}^d_+\}$ is an orthonormal basis for $\mathcal{L}^2(\R^d)$ given by the Hermite functions and $\langle\cdot,\cdot\rangle$ is the usual
inner product in $\mathcal{L}^2(\R^d)$. The Hermite-Sobolev spaces $\Sc_p, p \in \R$ are defined as the completion of $\Sc$ in
$\|\cdot\|_p$. Note that the dual space $\Sc_p^\prime$ is isometrically isomorphic with $\Sc_{-p}$ for $p\geq 0$. The following basic relations hold for the $\Sc_p$ spaces: for $0<q<p$, \[\Sc\subset\Sc_p\subset\Sc_q\subset\mathcal L^2(\R^d) = \Sc_0\subset\Sc_{-q}\subset\Sc_{-p}\subset\Sc^\prime.\]
We also have $\Sc = \bigcap_{p \geq 0}\Sc_p$ and  $\Sc^\prime = \bigcup_{p \geq 0}\Sc_{-p}$.

Consider the derivative maps denoted by $\partial_i:\Sc\to
\Sc$ for $i=1,\cdots,d$. We can extend these maps by duality to
$\partial_i:\Sc' \to \Sc'$ as follows: for $\psi_1 \in
\Sc'$,
\[\inpr{\partial_i \psi_1}{\psi_2}:=-\inpr{\psi_1}{\partial_i \psi_2}, \; \forall \psi_2
\in \Sc.\]
It is well-known that the derivative operators $\partial_i: \Sc_q \to \Sc_{q - \frac{1}{2}}, i = 1, 2, \cdots, d$ and $\partial^2_{ij}: \Sc_q \to \Sc_{q - 1}, i, j = 1, 2, \cdots, d$ are bounded linear operators for all $q \in \R$.

For $x \in \R^d$, let $\tau_x$ denote the translation operator on $\Sc$
defined by
$(\tau_x\psi)(z):=\psi(z-x), \, \forall z \in \R^d, \psi \in \Sc$. This operator can be
extended to $\tau_x:\Sc'\to \Sc'$ by
\[\inpr{\tau_x\psi_1}{\psi_2}:=\inpr{\psi_1}{\tau_{-x}\psi_2},\, \forall \psi_2 \in
\Sc.\]
Note that $\tau_x: \Sc_q \to \Sc_q$ is a bounded linear operator for any $x \in \R^d$ and any $q \in \R$ (see \cite[Theorem 2.1]{MR1999259}).

\subsection{Definitions and literature review}\label{S2:setup}

The initial condition $\phi$ of \eqref{main-spde} is in $\Sc^\prime = \bigcup_{q \geq 0}\Sc_{-q}$. Consequently, there exists $p \geq 0$ such that $\phi \in \Sc_{-p}$. In what follows, we work with this specific $p$ and assume that
\begin{equation}\label{assumption1}
\tag{Assumption 1}
\sigma_{ij}, b_i \in \Sc_p, \forall i, j = 1, 2, \cdots, d.
 \end{equation}
Using the duality between $\Sc_p$ and $\Sc_{-p}$, observe that $\sigma_{ij}$ and $b_i$'s are continuous linear functionals on $\Sc_{-p}$.\\
Let 
 \[B_{-p}(0, r):= \{y\in \Sc_{-p}:\, \|y\|_{-p}\leq r\}.\]
for $r > 0$. Then, for all $r > 0$
\begin{equation}
 C_1(r):= \max_{i,j}\sup_{y \in B_{-p}(0, r)} \{|\sigma_{ij}(y)|^2,\, |b_i(y)|\} < \infty.
 \end{equation}
By construction, $C_1(r)$ is non-decreasing in $r$. Consequently, the operators $L,A_j: \Sc_{-p}\to \Sc_{-p-1}, j = 1, 2, \cdots, d$ are bounded in the following sense:
\begin{equation}\label{L-A-bound}
 \|L(y)\|_{-p-1} \leq \tilde C_1(d, r) \|y\|_{-p}, \quad \|A_j(y)\|_{-p-1} \leq \tilde C_2(d, r) \|y\|_{-p}, \, \forall y \in B_{-p}(0, r)
 \end{equation}
for any $r > 0$. Here, $\tilde C_1(d, r)$ and $\tilde C_2(d, r)$ are some non-negative constants, depending on $d$ and $r$.

Let $\zeta$ be an arbitrary state, treated as an isolated point of $\hat \Sc_{-p} := \Sc_{-p} \cup \{\zeta\}$.

\begin{definition}[Local Strong solution]\label{defn-strngsoln}
A pair $(\{X_t\}_t, \eta)$ is called a local strong solution of \eqref{main-spde}, if the following conditions hold.
 \begin{enumerate}[label=(\roman*)]
     \item $X = \{X_t\}_t$ is an $\hat\Sc_{-p}$-valued continuous adapted process defined on the filtered probability space $(\Omega,\F,\{\F_t\}_{0\leq t\leq T},P)$,

     \item $\eta$ is an $\{\F_t\}_{0\leq t\leq T}$ stopping time with a.s. $X_t = \zeta, \forall \eta < t \leq T$,
     
     \item the following equality holds in $\Sc_{-p-1}$ a.s. for all $0 \leq t < \eta$.
 \begin{equation}\label{defn-strngSOLN11}
 X_t= \phi + \int_0^t L(X_s)\, ds + \int_0^t A(X_s)\cdot dB_s.
 \end{equation}
 \end{enumerate}
 \end{definition}

\begin{remark}
If for some local strong solution $(\{X_t\}_t, \eta)$, the equality \eqref{defn-strngSOLN11} holds a.s. for all $t \in [0, T]$, then we refer to $\{X_t\}_t$ as a (global) strong solution of \eqref{main-spde}.
\end{remark}

 \begin{definition}[Weak  solution]\label{defn-weaksoln} A system $\left((\Omega,\F,\{\F_t\}_{0\leq t\leq T},P), B, X\right)$ is called a weak solution of \eqref{main-spde}, where $B = \{B_t\}_{0\leq t\leq T}$ is a $d$-dimensional standard Brownian motion with respect to the filtration $\{\F_t\}_{0\leq t\leq T}$ and $X = \{X_t\}_{0\leq t\leq T}$ is an $\Sc_{-p}$-valued $\{\F_t\}_{0\leq t\leq T}$-adapted continuous process such that the equality \eqref{defn-strngSOLN11} holds in $\Sc_{-p-1}$ a.s. for all $t \in [0, T]$. 
 \end{definition}

\begin{remark}
If the filtered probability space $(\Omega,\F,\{\F_t\}_{0\leq t\leq T},P)$ is clear from the context, then for notational convenience, we shall write $(X, B)$ to denote the weak solution as mentioned in Definition \ref{defn-weaksoln}.
\end{remark}

In \cite{MR3063763}, the existence and uniqueness of local strong solutions to SPDE \eqref{main-spde} was considered. Consider the following SDE in $\R^d$
 \begin{equation}\label{sde1}
 Z_t = z + \int_0^t \bar\sigma (Z_s)\cdot dB_s + \int_0^t \bar b(Z_s)\, ds,
 \end{equation}
 where, $\bar\sigma = (\bar\sigma_{ij}) : \R^d \to \R^{d \times d}$ and $\bar b = (\bar b_i) :\R^d \to \R^d$ are defined by
 \begin{align}\label{rln-sigb-BRsigb}
 \begin{split}
 & \bar\sigma_{ij}(\rho):= \sigma_{ij}(\tau_\rho\phi),\\
 & \bar b_i(\rho) := b_i(\tau_\rho\phi).
 \end{split}
 \end{align}
for $\rho\in\R^d$. We make the following assumption.
\begin{equation}\label{assumption2}\tag{Assumption 2}
\text{The functions $\bar \sigma$ and $\bar b$ are locally Lipschitz.}
\end{equation}

\begin{theorem}[{\cite[Theorem 3.4 and Remark 3.9]{MR3063763}}]\label{exist-unq-nlspde}
Suppose \ref{assumption1} and \ref{assumption2} hold. Let $(\{Z_t\}_t, \eta)$ be the local strong solution of SDE \eqref{sde1} with the initial condition $z=0$. Then, $(\{X_t\}_t, \eta)$ is a local strong solution of the Stochastic PDE \eqref{main-spde}, where $X_t:=\tau_{Z_t}\phi, \forall t < \eta$.
 \end{theorem}


\section{Main Results}\label{S3:main-results}

Under the setup described in Subsection \ref{S2:setup}, we consider situations where the SPDE \eqref{main-spde}
has global strong solutions, i.e. for all time $t \in [0, T]$. Using the structure of $X_t = \tau_{Z_t}\phi$, it is enough to ensure the existence of global strong solutions $\{Z_t\}_t$ for the SDE \eqref{sde1}. Moreover, we require some norm-bounds on $\{X_t\}_t$ uniformly in time. We state the relevant assumptions below.
\begin{equation}\label{assumption3}\tag{Assumption 3}
\begin{split}
&\text{The SDE \eqref{sde1} does not explode in finite time and has a unique strong solution on the time}\\
&\text{interval $[0, T]$.}
\end{split}
\end{equation}
and 
\begin{equation}\label{assumption4}\tag{Assumption 4}
\text{There exists a constant $\lambda = \lambda(\{X_t\}_t) > 0$ such that $\Exp \sup_{t \in [0, T]}  \|X_t\|_{-p-1}^2 \leq \lambda$.}
\end{equation}

\begin{example}\label{assumption4-examples}
For completeness, we mention some examples where the above assumptions hold. For \ref{assumption3}, we refer to \cite{MR1121940, MR0247684, MR1876169}. For \ref{assumption4}, the following special cases may be considered.
\begin{enumerate}[label=(\alph*)]
    \item If $\phi = \delta_x$ for some $x \in \R^d$, then by \cite[Theorem 4.1]{MR2373102}, we have 
    \[\|X_t\|_{-p-1} = \|\delta_{x + Z_t}\|_{-p-1} \leq C_p\]
    where $C_p > 0$ is a constant depending only on $p$, provided $- p -1 < -\frac{d}{4}$ or equivalently, $p > \frac{d}{4} - 1$. \ref{assumption4} follows. Note that, taking $\phi$ as a finite linear combinations of $\delta_x, x \in \R^d$ also works.
    
    \item By \cite[Theorem 2.1]{MR1999259}, there exists a real polynomial $P_k$ of degree $k = 2([|p + 1|] + 1)$, such that 
    \[\|X_t\|_{-p-1} = \|\tau_{Z_t}\phi\|_{-p-1} \leq P_k(|Z_t|) \|\phi\|_{-p-1}.\]
    Without loss of generality, the coefficients of $P_k$ are taken to be non-negative.
    To have \ref{assumption4}, we need to work with those $\{Z_t\}_t$ such that $\Exp \sup_{t \in [0, T]} P_k(|Z_t|) < \infty$.
\end{enumerate}
\end{example}

Under \ref{assumption4}, we have $\{X_t\}_t \in \h\Hh_2$, where $\h\Hh_2$ denotes the space of adapted $\Sc_{-p-1}$-valued continuous stochastic processes $\{\varTheta_t\}_t$ satisfying
 \[ \sup_{0\leq t\leq T} \Exp\, \|\varTheta_t\|_{-p-1}^2<\infty.\]
 Note that $\h\Hh_2$ is a Banach space, with the norm (see \cite[p. 104]{MR2560625})
 \begin{equation}\label{defn-banach-spc-NRM}
 \|\varTheta\|_{\h\Hh_2}:= \left( \sup_{0\leq t\leq T} \Exp\, \|\varTheta_t\|_{-p-1}^2 \right)^{\frac{1}{2}}.
 \end{equation}

\begin{remark}
Markov property of the solutions $\{X_t\}_t$ has been discussed in Section 4 of \cite{MR3063763}.
\end{remark}

Let us consider the Stochastic PDE \eqref{main-spde} and look at the  exponential martingale, following Example 19.9, Chapter 19 of \cite{MR3012668}, viz.
 \begin{equation}\label{exp-mart1}
 M_t := \exp \left(\sum_{j=1}^d \left\{ \int_0^t h^j(s)\, dB^j_s - \frac{1}{2}\int_0^t h^j(s)^2\, ds\right\} \right),
 \end{equation}
where for $s \in [0, T]$,
\begin{equation}\label{dfnhjnw}
h^j(s) := \sqrt{\Exp\|X_s\|^2_{-p-1}},\ \ \forall \  j=1,2,\cdots,d.
\end{equation}
Using \ref{assumption4}, we conclude that the Novikov's condition holds, i.e.
\begin{equation}\label{Novikov-condition}
   \Exp \left[ \exp \left(\frac{1}{2}\int_0^T \sum_{j=1}^d h^j(s)^2\, ds\right) \right] = \Exp \left[ \exp \left(\frac{d}{2}\int_0^T \Exp\|X_s\|^2_{-p-1} \, ds\right) \right]<\infty. 
\end{equation}

 \begin{theorem}\label{grisanov}
 Consider the $\Sc_{-p}$-valued process $\{X_t\}_{0\leq t\leq T}$ satisfying \eqref{main-spde} in $\Sc_{-p-1}$ with
 \begin{equation}\label{grisanov-marting}
 \Exp \left[ \exp \left(\sum_{j=1}^d \left\{ \int_0^T \sqrt{\Exp\|X_s\|^2_{-p-1}}\, dB^j_s - \frac{1}{2}\int_0^T \Exp\|X_s\|^2_{-p-1}\, ds \right\} \right) \right] = 1.
 \end{equation}
 Then the process
 \begin{equation}\label{transformQ-Winer}
 \widehat{B^j_t} = B^j_t - \int_0^t \sqrt{\Exp\|X_s\|^2_{-p-1}}\, ds,\ \ \ t\in[0,T],\, \forall j=1,\cdots,d
 \end{equation}
 is a Brownian motion with respect to $Q$ on the probability space $(\Omega,\F,Q)$, where
 \begin{equation}\label{Q-BM-grisnv}
 dQ(\omega)= \exp \left(\sum_{j=1}^d \left\{ \int_0^T \sqrt{\Exp\|X_s\|^2_{-p-1}}\, dB^j_s - \frac{1}{2}\int_0^T \Exp\|X_s\|^2_{-p-1}\, ds \right\} \right)\, dP(\omega).
 \end{equation}

 \end{theorem}
 
 \begin{proof}
 Note that, $X_t$ is $\Sc_{-p}$ valued, $\Exp\|X_t\|^2_{-p-1}$ is finite and $\{B_t\}$ is a $d$-dimensional Brownian motion. Therefore the proof of Theorem \ref{grisanov} follows from the finite dimensional proof of Girsanov's theorem for SDEs, see \cite{MR2001996}.
 
 \end{proof}

 Consider the following two equations in $\Sc_{-p-1}$:
 \begin{align}
 & dX_t = L(X_t)\, dt + A(X_t) \cdot dB_t,\ \ \ X_0=\phi;\label{gris-spde11}\\
 & d\widetilde X_t = \left( L(\widetilde X_t)+ \hat L(t,\widetilde X_t)\right)\, dt + A(\widetilde X_t) \cdot dB_t,\ \ \ \widetilde X_0=\phi,\label{gris-spde00}
 \end{align}
 where $L$ and $A$ as in \eqref{defn-LA1} and
 \begin{equation}\label{defn-hatL}
  \hat L(t,y) :=  -  \sum_{j=1}^d h^j(t)\, A_j(y)  =   \sum_{i,j=1}^d  h^j(t)\,  \sigma_{ij}(y)\, \partial_i y,\ \forall y\in\Sc_{-p}.  
 \end{equation}
 Consider the following SDE:
 \begin{equation}\label{sde-grisnv-til}
 \widetilde Z_t^i := z^i +\int_0^t \sum_{j=1}^d \bar\sigma_{ij} (\widetilde Z_s)\, dB^j_s + \int_0^t \left( \bar b^i(\widetilde Z_s) - \sum_{j=1}^d h^j(s)\, \bar\sigma_{ij} (\widetilde Z_s) \right)\, ds,
 \end{equation}
 where,  $\bar\sigma_{ij}, \bar b_i$ are defined as in \eqref{rln-sigb-BRsigb}. 
 Note that, $\bar\sigma_{ij}, \bar b_i : \R\to\R$ and $\sigma_{ij}, b_i : \Sc_{-p}\to\R$.
 
 \begin{theorem}\label{gris-tran-inv-spdethm} Consider the SDE \eqref{sde-grisnv-til} with initial condition $z=0$. Then 
 $\widetilde X_t:=\tau_{\widetilde Z_t}\phi$ is an unique strong solution of SPDE \eqref{gris-spde00}.
 \end{theorem}
 
 \begin{proof}
  
 Applying It\^o's formula for the translation operator (\cite[Theorem 2.3]{MR1837298}), we have a.s.
 \begin{align*}
 & \tau_{\widetilde Z_t}\phi \\
 & = \phi - \sum_{i=1}^d \int_0^t \partial_i\,  \tau_{\widetilde Z_s}\phi\, d \widetilde Z^i_s + \frac{1}{2} \sum_{i,j=1}^d \int_0^t \partial^2_{ij}\,  \tau_{\widetilde Z_s}\phi\, d \left[\widetilde Z^i,\widetilde Z^j \right]_s \\
 & = \phi - \sum_{i=1}^d \int_0^t \partial_i\,  \tau_{\widetilde Z_s}\phi\, \left( \sum_{j=1}^d \bar\sigma_{ij} (\widetilde Z_s) dB^j_s \right) \\
 & \quad - \sum_{i=1}^d \int_0^t \partial_i\,  \tau_{\widetilde Z_s}\phi\, \left( \bar b^i(\widetilde Z_s) -   \sum_{j=1}^d h^j(s)\, \bar\sigma_{ij} (\widetilde Z_s) \right)\, ds \\
 & \quad + \frac{1}{2} \sum_{i,j=1}^d \int_0^t \partial^2_{ij}\,  \tau_{\widetilde Z_s}\phi\, \left( \bar\sigma (\widetilde Z_s)\, \bar\sigma^t (\widetilde Z_s)\right)_{ij} ds \\
 & = \phi - \sum_{i=1}^d \int_0^t \partial_i\,  \tau_{\widetilde Z_s}\phi\, \left( \sum_{j=1}^d \sigma_{ij} \Big(\tau_{\widetilde Z_s}\phi\Big)\,  dB^j_s \right) \\
 & \quad - \sum_{i=1}^d \int_0^t \partial_i\,  \tau_{\widetilde Z_s}\phi\, \left( b^i \Big(\tau_{\widetilde Z_s}\phi\Big) -   \sum_{j=1}^d h^j(s)\,  \sigma_{ij}  \Big(\tau_{\widetilde Z_s}\phi\Big) \right) \, ds \\
 & \quad + \frac{1}{2} \sum_{i,j=1}^d \int_0^t \partial^2_{ij}\,  \tau_{\widetilde Z_s}\phi\, \left(\sigma \sigma^t \right)_{ij} \Big(\tau_{\widetilde Z_s}\phi\Big)\, ds.
 \end{align*}
 Therefore, $\widetilde X_t:=\tau_{\widetilde Z_t}\phi$ is a solution of SPDE \eqref{gris-spde00}. The uniqueness of $\{\widetilde X_t\}_{t \in [0, T]}$ as a solution to \eqref{gris-spde00} follows from uniqueness of the SDE for $\{\widetilde Z_t\}_t$, as $\bar\sigma_{ij},\, \bar b_i$ are locally Lipschitz. 
 \end{proof}

\begin{theorem}\label{final-girsanov}
Consider $\hat L$ as in \eqref{defn-hatL}, the probability measure $Q$ as in \eqref{Q-BM-grisnv} and the $Q$-Brownian motion $\widehat B$ as in \eqref{transformQ-Winer}. Note that the Novikov condition \eqref{Novikov-condition} holds. Then the $\Sc_{-p}$ valued process $\{\widetilde X_t\}$ as in \eqref{gris-spde00}) is a solution to
\begin{equation}\label{SPDE-new-BM}
d\widetilde X_t = L(\widetilde X_t)\, dt + A(\widetilde X_t) \cdot d \widehat B_t,\ \ \ \widetilde X_0=\phi
\end{equation}
and has the same law under $Q$ as $\{X_t\}_t$ in \eqref{gris-spde11} under $P$.

\end{theorem}

\begin{proof}
$\{\widetilde Z_t\}_t$ satisfies the SDE 
\begin{equation}\label{SDE-Z-new-BM}
\widetilde Z_t^i = z^i +\int_0^t \sum_{j=1}^d \bar\sigma_{ij} (\widetilde Z_s)\, d \widehat B^j_s + \int_0^t \bar b^i(\widetilde Z_s)\, ds,\, i = 1, \cdots, d
\end{equation}
under $Q$ (see \eqref{sde-grisnv-til}). Hence, its law under $Q$ is the same as that of $\{Z_t\}_t$ (under $P$) satisfying
\begin{equation}
Z_t^i = z^i +\int_0^t \sum_{j=1}^d \bar\sigma_{ij} (Z_s)\, d B^j_s + \int_0^t \bar b^i(Z_s)\, ds,\, i = 1, \cdots, d
\end{equation}
under $P$ (\cite[Theorem 8.6.5]{MR2001996}). Since $P$ and $Q$ are equivalent probability measures and $ X_t:=\tau_{ Z_t}\phi$ $P$-a.s., $\widetilde X_t:=\tau_{\widetilde Z_t}\phi$ $Q$-a.s., we have the result.
\end{proof}
 
\begin{remark}\label{final-girsanov21}
Under the correspondence between the SPDE \eqref{main-spde} and the SDE \eqref{sde1}, the same Brownian motion appears in both the equations. We are, therefore, able to use the finite dimensional Girsanov theorem in our arguments and the new Brownian motion appears again in both the equations \eqref{SPDE-new-BM} and \eqref{SDE-Z-new-BM}. It is noteworthy that the same Novikov condition is used in changing the Brownian motion for the Stochastic PDEs, as well as the SDEs. Note that the condition is in terms of the solutions of the SPDE \eqref{main-spde}.
\end{remark} 

 \section{Applications}\label{S4:Applications}

In this section, we apply our main results, Theorems \ref{grisanov}, \ref{gris-tran-inv-spdethm} and \ref{final-girsanov} in the following two examples,  to construct weak solutions. Though the examples are described in 1-dimension for simplicity, they can be extended to any general $d$-dimensions in a similar fashion. 

  \begin{example}\label{example1}
Consider $Z_t = B_t, \forall t \in [0, T]$. This process $\{Z_t\}_t$ can be thought of as the solution to the following SDE
  \begin{equation}\label{sde-example1}
  dZ_t=dB_t,\ \ Z_0=0,\ \text{for}\ t\in[0,T].
  \end{equation}
 Take $X_t:=\delta_{B_t}=\tau_{B_t}\delta_0$, is the solution of the following SPDE in $\Sc'$
  \begin{equation}\label{sPde-example1}
  \delta_{B_t}= \delta_0 - \int_0^t\partial\delta_{B_s}\,dB_s + \frac{1}{2}\int_0^t\partial^2\delta_{B_s}\, ds.
  \end{equation}
   For any $p>\frac{d}{4}$, $\sup_{t \in [0, T]}\|\delta_{B_t}\|_{-p}\leq C < \infty$, for some constant $C = C(p, d) > 0$ see \cite[Theorem 4.1]{MR2373102} (also see the comments in Example \ref{assumption4-examples} above). Then, $\delta_{B_t}$ is $\Sc_{-p}$-valued, for $p>\frac{d}{4}$, whereas equation \eqref{sPde-example1} holds in $\Sc_{-p-1}$ and our Novikov condition \eqref{Novikov-condition} will be
  \begin{equation}\label{exmplnovkov1}
  \Exp \left[ \exp \left(\frac{1}{2}\int_0^T \Exp\, \|\delta_{B_s}\|^2_{-p-1} \, ds\right) \right]\leq \Exp \left[ \exp \left(\frac{1}{2}\int_0^T \Exp\, \|\delta_{B_s}\|^2_{-p} \, ds\right) \right] <\infty.
  \end{equation}
 Note that $h(t):= \sqrt{\Exp\, \|\delta_{B_t}\|^2_{-p-1}}$ (see \eqref{dfnhjnw}) and by \eqref{transformQ-Winer}, the new Brownian motion is given by
  \begin{equation}\label{newBM-exmp1}
  \widehat B_t = B_t - \int_0^t \sqrt{\Exp\|\delta_{B_s}\|^2_{-p-1}}\, ds.
  \end{equation}
  Note that, \eqref{exmplnovkov1} is a  sufficient condition for \eqref{grisanov-marting} to hold, see \cite[Proposition 10.17]{MR3236753}. Now, consider the SDE:
  \begin{equation}\label{sde2-example1}
  d\widetilde Z_t = d\widehat B_t + \sqrt{\Exp\|\delta_{B_t}\|^2_{-p-1}}\, dt, \ \ Z_0=0,
  \end{equation}
  Now, by It\^o's formula for the translation operator, as applied in Theorem \ref{gris-tran-inv-spdethm}
\begin{align}\label{spde2-example1}
\begin{split}
\widetilde X_t:=\tau_{\widetilde Z_t}\delta_0 & = \delta_{\widetilde Z_t} \\
& = \delta_0 - \int_0^t \partial \delta_{\widetilde Z_s}\, d\widetilde Z_s + \frac{1}{2}  \int_0^t \partial^2\,  \delta_{\widetilde Z_s} \, ds\\
& = \delta_0 - \int_0^t \partial \delta_{\widetilde Z_s}\, \left( d\widehat B_s + \sqrt{\Exp\|\delta_{B_s}\|^2_{-p-1}}\, ds\right) + \frac{1}{2}  \int_0^t \partial^2\,  \delta_{\widetilde Z_s} \, ds \\
& = \delta_0  + \int_0^t \left( \frac{1}{2} \partial^2\, \delta_{\widetilde Z_s} - \sqrt{\Exp\|\delta_{B_s}\|^2_{-p-1}}\, \partial \delta_{\widetilde Z_s} \right)\, ds - \int_0^t \partial \delta_{\widetilde Z_s}\, d\widehat B_s.
\end{split}
\end{align}
Observe that, $(\widetilde Z,\widehat B)$ from \eqref{sde2-example1} is a weak solution of \eqref{sde-example1} and $(\widetilde X,\widehat B)$ from \eqref{spde2-example1} is a weak solution of \eqref{sPde-example1} and the solutions are weakly unique in the following sense:
\[\mathcal{L}(Z)= \mathcal{L}(\widetilde Z),\ \ \text{and}\ \ \mathcal{L}(X)= \mathcal{L}(\widetilde X),\]
where $\mathcal{L}(\cdot)$ denotes the law of a stochastic process. Note that, the same Novikov condition \eqref{exmplnovkov1} and new Brownian motion \eqref{newBM-exmp1} are used to construct the weak solutions of SDEs and SPDEs. 

\end{example}

\begin{example}
Consider $Z_t = B_t^2, \forall t \in [0, T]$. This $\{Z_t\}_t$ can be thought of as the solution to the following SDE
  \begin{equation}\label{sde-example12}
  dZ_t=d\left(B_t^2\right),\ \ Z_0=0,\ \text{for}\ t\in[0,T].
  \end{equation}
Applying It\^o's formula, we have
\begin{equation}\label{Zt1st2ndexmpl}
dZ_t=d\left(B_t^2\right) = 2B_t\, dB_t+ dt.
\end{equation}
Applying It\^o's formula for the translation operator (\cite[Theorem 2.3]{MR1837298}), we have a.s.
\begin{align}\label{spdetrnsfrm2ndexmp1}
\begin{split}
X_t:= \tau_{Z_t}\delta_0 & = \delta_{B_t^2} \\
& = \delta_0 - \int_0^t\partial\delta_{B_s^2}\, dZ_s + \frac{1}{2}\int_0^t\partial^2\delta_{B_s^2}\, d[Z,Z]_s \\
& = \delta_0 - \int_0^t\partial\delta_{B_s^2}\, \left( 2B_s\, dB_s+ ds \right) + \int_0^t\partial^2\delta_{B_s^2}\, 2B_s^2\, ds \\
& = \delta_0 + \int_0^t \left( 2B_s^2\,  \partial^2\delta_{B_s^2} - \partial\delta_{B_s^2} \right)\, ds - \int_0^t 2B_s\, \partial\delta_{B_s^2}\, dB_s
\end{split}
\end{align}
 Similar to Example \ref{example1}, $\delta_{B_t^2}$ is $\Sc_{-p}$-valued, for $p>\frac{d}{4}$, whereas equation \eqref{spdetrnsfrm2ndexmp1} holds in $\Sc_{-p-1}$ and our Novikov condition of \eqref{Novikov-condition} will be 
\begin{equation}\label{novi-example2}
\Exp \left[ \exp \left(\frac{1}{2}\int_0^T \Exp\, \|\delta_{B_s^2}\|^2_{-p-1} \, ds\right) \right] \leq \Exp \left[ \exp \left(\frac{1}{2}\int_0^T \Exp\, \|\delta_{B_s^2}\|^2_{-p} \, ds\right) \right] <\infty,
\end{equation}
Here, $h(t):= \sqrt{\Exp\, \|\delta_{B_t^2}\|^2_{-p-1}}$ (see \eqref{dfnhjnw}) and by \eqref{transformQ-Winer}, the new Brownian motion is given by
  \begin{equation}\label{2exmpltrnsfrmBt}
  \widehat B_t = B_t - \int_0^t \sqrt{\Exp\|\delta_{B_s^2}\|^2_{-p-1}}\, ds.
  \end{equation}
  From \eqref{Zt1st2ndexmpl}, substituting $dB_t$ of \eqref{2exmpltrnsfrmBt}, we obtain 
 \begin{align}\label{ztild2ndexmpl22}
 \begin{split}
 d\widetilde Z_t & = 2B_t \left( d\widehat B_t + \sqrt{\Exp\|\delta_{B_t^2}\|^2_{-p-1}}\, dt\right) + dt \\
 & = 2B_t\, d\widehat B_t + \left( 2B_t  \sqrt{\Exp\|\delta_{B_t^2}\|^2_{-p-1}} +1  \right)\, dt.
 \end{split}
 \end{align}
 Now, by It\^o's formula for the translation operator
 \begin{align}\label{fnltrnsfrmweksoln2nd}
\begin{split}
& \widetilde X_t:= \delta_{\widetilde Z_t} \\
& = \delta_0 - \int_0^t \partial \delta_{\widetilde Z_s}\, d\widetilde Z_s + \frac{1}{2}  \int_0^t \partial^2\,  \delta_{\widetilde Z_s} \, d[\widetilde Z, \widetilde Z]_s\\
& = \delta_0 - \int_0^t \partial \delta_{\widetilde Z_s} \left\{ 2B_s\, d\widehat B_s + \left( 2B_s  \sqrt{\Exp\|\delta_{B_s^2}\|^2_{-p-1}} +1  \right)\, ds\right\} \\
& \quad + \int_0^t  2B_s^2\,  \partial^2\,  \delta_{\widetilde Z_s} \, ds \\
& = \delta_0 + \int_0^t \left\{ 2B_s^2\,  \partial^2\,  \delta_{\widetilde Z_s} 
 - \left( 2B_s  \sqrt{\Exp\|\delta_{B_s^2}\|^2_{-p-1}} +1  \right)\partial \delta_{\widetilde Z_s}  \right\}\, ds \\
 & \quad - \int_0^t 2B_s\, \partial \delta_{\widetilde Z_s}\, d\widehat B_s.
\end{split}
 \end{align}
Similar to \ref{example1}, $(\widetilde Z,\widehat B)$ from \eqref{ztild2ndexmpl22} is a weak solution of \eqref{sde-example12} and $(\widetilde X,\widehat B)$ from \eqref{fnltrnsfrmweksoln2nd} is a weak solution of \eqref{spdetrnsfrm2ndexmp1} and the solutions   are weakly unique as they are equal in law, i.e.
\[\mathcal{L}(Z)= \mathcal{L}(\widetilde Z),\ \ \text{and}\ \ \mathcal{L}(X)= \mathcal{L}(\widetilde X).\]
Here also, the same Novikov condition \eqref{novi-example2} and new Brownian motion \eqref{2exmpltrnsfrmBt} are used to construct the weak solutions of SDEs and SPDEs.

\end{example}

\textbf{Acknowledgement:} Suprio Bhar was partially supported by the INSPIRE Faculty Award DST/INSPIRE/04/
2017/002835 (Department of Science and Technology, Government of India). Barun Sarkar acknowledges the support of SERB project -  SRG/2022/000991, Government of India.

 \bibliographystyle{amsplain}
 \bibliography{references}
 
 \end{document}